\documentclass[11pt]{article}

\usepackage[T1]{fontenc}
\usepackage[utf8]{inputenc}
\usepackage{amsmath,amssymb,amsthm,mathtools}
\usepackage[margin=1in]{geometry}
\usepackage{microtype}
\usepackage{enumitem}
\usepackage{listings}
\usepackage[colorlinks=true,linkcolor=blue,citecolor=blue,urlcolor=blue]{hyperref}

\newtheorem{theorem}{Theorem}
\newtheorem{lemma}{Lemma}
\newtheorem{proposition}{Proposition}
\usepackage[capitalise,nameinlink]{cleveref}

\newcommand{\R}{\mathbb R}
\newcommand{\dd}{\,d}
\newcommand{\calH}{\mathcal H}

\def\isarxiv{1}

\begin{document}

\title{A Counterexample to the Gaussian Completely Monotone Conjecture}
\ifdefined\isarxiv
\author{Yuzhou Gu\thanks{sevenkplus.g@gmail.com} \and Mark Sellke\thanks{msellke@fas.harvard.edu}}
\else
\author{}
\fi
\date{}
\maketitle

\begin{abstract}
We provide an explicit probability measure on \(\R\) for which the fifth time derivative of the entropy along the heat flow is positive at some time. This disproves the Gaussian completely monotone (GCM) conjecture (Cheng--Geng '15) and therefore also the Gaussian optimality conjecture (McKean '66) and the entropy power conjecture (Toscani '15).
Our proof also implies the existence of a \emph{log-concave} probability measure on \(\R\) for which the GCM conjecture fails at some order. The explicit counterexample was found by GPT-5.5 Pro.
\end{abstract}

\section{Introduction} \label{sec:intro}
Let \(P_t=e^{t\partial_x^2}\) denote the heat semigroup on \(\R\). For a real variable \(X\) with law \(\mu\), \(P_t \mu\) is the law of \(X+\sqrt{2t} G\), where \(G\) is a standard Gaussian independent of \(X\). For \(t>0\), this measure is represented by the smooth function
\[
   p_t*\mu,
   \qquad
   p_t(x)=(4\pi t)^{-1/2}e^{-\frac{x^2}{4t}}.
\]
We write
\[
   H_\mu(t)=\calH(P_t\mu)
   =\int_\R (p_t*\mu)(x)\log (p_t*\mu)(x)\,\dd x.
\]
The Gaussian completely monotone (GCM) conjecture \cite{cheng2015higher} states that
\begin{equation}
\label{eq:GCM}
   (-1)^m H_\mu^{(m)}(t)\ge0,
   \qquad m\ge1,t>0.
\end{equation}
A strengthening of this conjecture is the Gaussian optimality conjecture due to \cite[Section 12]{mckean1966speed}, which states that the left-hand side of \eqref{eq:GCM} is exactly minimized by Gaussian measures when \(\mu\) has fixed variance.
A further strengthening is the entropy power conjecture \cite{toscani2015concavity}, which asserts that
\[
(-1)^{m-1}\frac{d^m}{dt^m} N_\mu(t)\geq 0,
\qquad
N_\mu(t)=e^{-2H_\mu(t)}.
\]
The same conjectures extend to higher dimension \(d>1\) (with an additional factor of \(1/d\) in the exponent of the last expression). It is easy to see that McKean's conjecture is stronger than GCM, while \cite{wang2024entropy} showed that the entropy power conjecture implies McKean's conjecture.

In dimension \(1\), the GCM conjecture has been proved for derivatives of order \(m\le 4\) \cite{cheng2015higher}.
The higher-dimensional version is true up to order \(m\le 2\) via Costa's entropy power inequality \cite{costa1985new,villani2000short}, and in dimension \(d\le 4\) for \(m\le 3\) \cite{guo2022lower}.
The special case of log-concave probability measures has also attracted considerable interest \cite{toscani2015concavity,zhang2018gaussian,guo2022lower,wang2025higher}.
We refer the reader to \cite{ledoux2024differentials,wang2025higher} and references therein for further discussion of the history of these conjectures.

Our main result is that in dimension \(1\), the GCM conjecture fails at the fifth order, which is optimal by the result of \cite{cheng2015higher} mentioned above. Thus all three related conjectures are false.

\begin{theorem}\label{thm:gcmc}
There exists a probability measure \(\mu\) on \(\R\) and a time \(t_0>0\) such that
\[
   H_\mu^{(5)}(t_0)>0.
\]
\end{theorem}

For log-concave probability measures, we show that the GCM conjecture also fails.
\begin{theorem}\label{thm:gcmc-log-conc}
There exists a log-concave probability measure \(\mu\) on \(\R\), a time \(t_0>0\), and an integer \(m\ge 1\), such that
\[
   (-1)^m H_\mu^{(m)}(t_0)<0.
\]
\end{theorem}
An explicit order is not obtained, but it must be at least \(6\), because the GCM conjecture is known to hold for log-concave probability measures up to the fifth order \cite{zhang2018gaussian,wang2025higher}.

\ifdefined\isarxiv
\paragraph{Acknowledgements} We thank Mehtaab Sawhney for helpful discussions.
\fi

\section{Proof of \texorpdfstring{\cref{thm:gcmc}}{Theorem 1}} \label{sec:proof-thm-gcmc}

Let \(\mu\) be the probability measure
\begin{equation}\label{eq:atomic-counterexample}
   \mu
   =
   0.002\,\delta_0
   +
   \sum_{j=0}^{7}
   c_j
   \left(\delta_{1.8+j}+\delta_{-(1.8+j)}\right),
\end{equation}
where
\begin{equation}\label{eq:c-vector}
   c=(0.034,\ 0.093,\ 0.134,\ 0.123,\ 0.075,\ 0.031,\ 0.008,\ 0.001).
\end{equation}
For \(t>0\), write
\[
   g_t=p_t*\mu,
   \qquad
   H_\mu(t)=\int_\R g_t(x)\log g_t(x)\,\dd x .
\]
Since \(\mu\) has finite support, \(g_t\) is a finite Gaussian mixture; in particular, \(H_\mu(t)\in\R\) for every \(t>0\).

\begin{proposition}\label{prop:certified-fifth}
For \(\mu\) defined by \eqref{eq:atomic-counterexample},
\[
   0.36 < H_\mu^{(5)}(1/3) < 0.37 .
\]
In particular, \(H_\mu^{(5)}(1/3)>0\).
\end{proposition}
Clearly, \cref{thm:gcmc} follows from \cref{prop:certified-fifth}.

\begin{proof}[Proof of \cref{prop:certified-fifth}]
Let the atoms of \(\mu\) be
\[
   a_0=0,
   \qquad
   a_{j,+}=1.8+j,
   \qquad
   a_{j,-}=-(1.8+j),
   \qquad 0\le j\le 7,
\]
with corresponding weights
\[
   w_0=0.002,
   \qquad
   w_{j,+}=w_{j,-}=c_j.
\]
Define
\[
   \rho_i(x,t)
   =w_i(4\pi t)^{-1/2}
      \exp\left(-\frac{(x-a_i)^2}{4t}\right),
   \qquad
   g_t(x)=\sum_i\rho_i(x,t).
\]
We use the physicists' Hermite polynomials \(\mathsf H_k\), given by
\[
   \frac{d^k}{dz^k}e^{-z^2}=(-1)^k\mathsf H_k(z)e^{-z^2}.
\]
Since \(p_t\) solves \(\partial_t p_t=\partial_x^2p_t\), for every \(n\ge1\),
\begin{equation}\label{eq:heat-hermite}
   \partial_t^n\rho_i(x,t)
   =
   \mathsf H_{2n}\!\left(\frac{x-a_i}{2\sqrt t}\right)
   \frac{\rho_i(x,t)}{(4t)^n}.
\end{equation}
Set
\[
   \pi_i(x,t)=\frac{\rho_i(x,t)}{g_t(x)},
   \qquad
   K_n(x,t)=\frac{\partial_t^n g_t(x)}{g_t(x)}.
\]
Then \(K_0=1\), and \eqref{eq:heat-hermite} gives
\begin{equation}\label{eq:K-formula}
   K_n(x,t)
   =
   \sum_i 
   \frac{\pi_i(x,t)}{(4t)^{n}}   \mathsf H_{2n}\!\left(\frac{x-a_i}{2\sqrt t}\right),
   \qquad n\ge1 .
\end{equation}
Next put
\[
   L_n(x,t)=\partial_t^n\log g_t(x).
\]
Thus \(L_0=\log g_t\). Since \(g_t=e^{\log g_t}\), differentiating \(n\) times and solving for \(L_n\) yields, for \(n\ge1\),
\begin{equation}\label{eq:L-recursion}
   L_n
   =
   K_n
   -
   \sum_{r=1}^{n-1}
   \binom{n-1}{r-1}L_rK_{n-r}.
\end{equation}
Finally, by Leibniz's rule,
\begin{equation}\label{eq:G-def}
   \partial_t^5(g_t\log g_t)
   =
   g_t
   \sum_{r=0}^{5}\binom{5}{r}K_rL_{5-r}.
\end{equation}
At \(t_0=1/3\), define
\begin{equation}\label{eq:Gx-def}
   G(x)=
   g_{t_0}(x)
   \sum_{r=0}^{5}\binom{5}{r}K_r(x,t_0)L_{5-r}(x,t_0).
\end{equation}
All derivatives in \eqref{eq:G-def} are bounded by a polynomial in \(|x|\) times a Gaussian tail, so differentiating under the integral sign is justified and
\begin{equation}\label{eq:H5-integral}
   H_\mu^{(5)}(t_0)=\int_\R G(x)\,\dd x.
\end{equation}
The measure \(\mu\) is symmetric, so \(G\) is even and
\[
   H_\mu^{(5)}(t_0)=2\int_0^\infty G(x)\,\dd x.
\]

It remains to certify this one-dimensional integral. For each interval \(I\), let
\[
   C_I=2\int_I G(x)\,\dd x.
\]
The following table records the finite part of a ball-arithmetic certificate. Each displayed value \(M_I\) is rounded to four decimal places, and the SageMath/Arb computation in \cref{app:verification-code}, using exact rational input data and 256-bit complex balls, proves
\begin{equation}\label{eq:table-cert-radius}
   C_I\in[M_I-10^{-4},M_I+10^{-4}].
\end{equation}
The computation evaluates the exact formulae \eqref{eq:K-formula}, \eqref{eq:L-recursion}, and \eqref{eq:Gx-def}.
\[
\begin{array}{c|r@{\qquad}c|r}
I & M_I & I & M_I \\
\hline
$[0,0.5]$ & 161.9975 & $[5.5,6]$ & -3.3394\\
$[0.5,1]$ & -219.0117 & $[6,6.5]$ & 29.4837\\
$[1,1.5]$ & 73.0107 & $[6.5,7]$ & -45.4866\\
$[1.5,2]$ & 1.8755 & $[7,7.5]$ & 50.1159\\
$[2,2.5]$ & -49.7406 & $[7.5,8]$ & -43.9915\\
$[2.5,3]$ & 78.9371 & $[8,8.5]$ & 29.8270\\
$[3,3.5]$ & -92.4682 & $[8.5,9]$ & -14.0523\\
$[3.5,4]$ & 90.8184 & $[9,10]$ & 5.7303\\
$[4,4.5]$ & -79.1908 & $[10,12]$ & -3.1009\\
$[4.5,5]$ & 58.8918 & $[12,15]$ & 0.3030\\
$[5,5.5]$ & -30.2433 & $[15,25]$ & -0.0000
\end{array}
\]
We now bound the omitted tail. Since the support of \(\mu\) is contained in \([-9,9]\), for \(x\ge25\) we have
\[
   g_{t_0}(x)
   \le
   (4\pi t_0)^{-1/2}
   \exp\left(-\frac{(x-9)^2}{4t_0}\right)
   \le
   \exp\left(-\frac34(x-9)^2\right).
\]
Let
\[
   Q(x)=\sum_{r=0}^{5}\binom{5}{r}K_r(x,t_0)L_{5-r}(x,t_0),
   \qquad G(x)=g_{t_0}(x)Q(x).
\]
For \(x\ge25\), the inequalities \(|a_i|\le8.8\) and
\[
   \left|\frac{x-a_i}{2\sqrt{t_0}}\right|
   \le 1.13(1+x)
\]
combined with the explicit Hermite polynomials up to degree ten give
\[
   |K_n(x,t_0)|\le A_n(1+x)^{2n},
   \qquad
   A=(1,6,56,825,15472,350000),
   \qquad 0\le n\le5.
\]
Also, \(|L_0(x,t_0)|=|\log g_{t_0}(x)|\le2(1+x)^2\) on
\([25,\infty)\). Indeed, the upper bound above gives \(g_{t_0}(x)<1\), while
the atom of mass \(10^{-3}\) at \(8.8\) gives
\[
   g_{t_0}(x)\ge
   10^{-3}(4\pi t_0)^{-1/2}
   \exp\left(-\frac{(x-8.8)^2}{4t_0}\right),
   \qquad x\ge25.
\]
Since \(t_0=1/3\), this implies
\[
   -\log g_{t_0}(x)
   \le
   \log\!\left(10^3(4\pi t_0)^{1/2}\right)
   +\frac34(x-8.8)^2
   \le 2(1+x)^2,
   \qquad x\ge25.
\]
Applying the recursion \eqref{eq:L-recursion} gives, for \(1\le n\le5\),
\[
   |L_n(x,t_0)|\le B_n(1+x)^{2n},
   \qquad
   B=(6,92,2265,76648,3347024).
\]
Consequently,
\[
   |Q(x)|\le 10^8(1+x)^{12},
   \qquad x\ge25.
\]
Finally, \((1+x)^{12}\le \exp((x-9)^2/4)\) for \(x\ge25\), and the elementary Gaussian tail bound gives
\begin{align}
   \left|2\int_{25}^{\infty}G(x)\,\dd x\right|
   &\le
   2\cdot10^8\int_{25}^{\infty}e^{-\frac12(x-9)^2}\,\dd x  \\
   &\le
   2\cdot10^8\frac{e^{-128}}{16}
   <10^{-20}. \label{eq:tail-cert}
\end{align}
Adding the intervals in the table, using \eqref{eq:table-cert-radius} for the \(22\) finite intervals and \eqref{eq:tail-cert} for the tail, gives
\[
   0.36
   <
   H_\mu^{(5)}(1/3)
   <
   0.37.
   \qedhere
\]
\end{proof}

\section{Proof of \texorpdfstring{\cref{thm:gcmc-log-conc}}{Theorem 2}} \label{sec:proof-thm-gcmc-log-conc}

We now show that the preceding counterexample leads to a log-concave counterexample, at a higher order. The additional ingredient is backward propagation of complete monotonicity for analytic functions.

\begin{lemma}[Backward propagation of complete monotonicity]\label{lem:backward}
Let \(A>0\), and let \(F\) be real analytic on \((-A,\infty)\). Suppose that
\[
   (-1)^mF^{(m)}(s)\ge0,\qquad m\ge1,\ s>0.
\]
Then the same inequalities hold for all \(m\ge1\) and all \(s\in(-A,\infty)\).
\end{lemma}

\begin{proof}
This is a standard consequence of the local Taylor-series characterization of
completely monotone functions; see, e.g., \cite{widder1941laplace,schilling2012bernstein}.
For completeness, we recall the short argument. Set \(U=-F'\). Then \(U\) is
completely monotone on \((0,\infty)\), and all inequalities
\((-1)^nU^{(n)}\ge0\) extend to \(0\) by continuity. If they hold at a point
\(c\), then the Taylor expansion at \(c\) gives, for \(h\ge0\) sufficiently small,
\[
   (-1)^nU^{(n)}(c-h)
   =
   \sum_{j=0}^{\infty}
   \frac{(-1)^{n+j}U^{(n+j)}(c)}{j!}h^j
   \ge0 .
\]
Thus the inequalities propagate locally to the left; a standard connectedness
argument propagates them throughout \((-A,0]\).
\end{proof}


\begin{proof}[Proof of \cref{thm:gcmc-log-conc}]
Let \(\mu\) be the finitely supported probability measure in \eqref{eq:atomic-counterexample}.
We use two elementary facts. First, if a probability measure is supported in \([-R,R]\), then its heat evolution is strictly log-concave for every \(t>R^2/2\). Indeed, writing \(\pi_x\) for the posterior law of the initial atom conditioned on the endpoint \(x\), one has
\[
   \partial_x^2\log(p_t*\mu)(x)
   =-\frac{1}{2t}+\frac{\operatorname{Var}_{\pi_x}(a)}{4t^2}
   \le -\frac{1}{2t}+\frac{R^2}{4t^2}<0.
\]
Second, for a finitely supported initial measure, \(H_\mu(t)\) is real
analytic on \((0,\infty)\). Indeed, locally uniformly for \(t\) in a compact
subinterval of \((0,\infty)\), the finite Gaussian mixture \(g_t\) and all its
\(t\)-derivatives are bounded by a polynomial in \(|x|\) times \(e^{-c x^2}\).
Moreover \(g_t(x)\) has matching Gaussian lower bounds coming from the extreme
atoms, so \(|\log g_t(x)|\le C(1+x^2)\). Thus the Taylor expansion in \(t\) may
be integrated termwise by Gaussian domination.

Take \(T\) large enough so that
\(
   \nu=P_T\mu
\)
is a log-concave probability measure.
Suppose, for contradiction, that the GCM conjecture holds for \(\nu\). Then
\[
   (-1)^m\frac{d^m}{ds^m}\calH(P_s\nu)\ge0,
   \qquad m\ge1,
   \qquad s>0.
\]
But \(P_s\nu=P_{T+s}\mu\). Thus the function
\[
   F(s)=H_\mu(T+s)
\]
is real analytic on \((-T,\infty)\) and satisfies
\[
   (-1)^mF^{(m)}(s)\ge0,
   \qquad m\ge1,
   \ s>0.
\]
By \cref{lem:backward}, the same inequalities hold for all \(m\ge1\) and all \(s>-T\). Taking \(s=1/3-T\) and \(m=5\) gives
\[
   0\le (-1)^5F^{(5)}(1/3-T)=-H_\mu^{(5)}(1/3),
\]
which contradicts \cref{prop:certified-fifth}.
\end{proof}

\bibliographystyle{alpha}
\bibliography{ref}

\appendix
\section{Verification code}\label{app:verification-code}

The following SageMath script gives the rigorous verification used in
\cref{prop:certified-fifth}. The input atoms and weights are exact rationals.
All integrations are carried out with Arb ball arithmetic \cite{johansson2017arb}
through SageMath's \texttt{ComplexBallField.integral}; the printed balls are
rigorous enclosures.

\begin{lstlisting}[language=Python]
from sage.all import ComplexBallField, RealBallField, QQ, binomial

# Increasing PREC tightens the final balls but does not change the input data.
PREC = 256
C = ComplexBallField(PREC)
R = RealBallField(PREC)

tau = QQ(1) / QQ(3)
tau_C = C(tau)
normalization = 1 / (4 * C.pi() * tau_C).sqrt()

# Exact rational atoms and weights for
# 0.002 delta_0 + sum_j c_j (delta_{1.8+j} + delta_{-(1.8+j)}).
data = [(QQ(0), QQ(2) / QQ(1000))]
c = [34, 93, 134, 123, 75, 31, 8, 1]
for j, cj in enumerate(c):
    a = QQ(18) / QQ(10) + QQ(j)
    w = QQ(cj) / QQ(1000)
    data.append((-a, w))
    data.append(( a, w))
data.sort(key=lambda z: z[0])

assert sum(w for _, w in data) == 1


def hermite_phys(n, z):
    """Physicists' Hermite polynomial H_n(z), evaluated by recursion."""
    parent = z.parent()
    if n == 0:
        return parent(1)
    H0 = parent(1)
    H1 = 2 * z
    if n == 1:
        return H1
    for k in range(1, n):
        H0, H1 = H1, 2 * z * H1 - 2 * k * H0
    return H1


def G(z, analytic):
    """Ball enclosure for d^5/dt^5 (g_t log g_t) at t = 1/3."""
    rhos = []
    for a, w in data:
        y = z - C(a)
        rho = C(w) * normalization * (-(y * y) / (4 * tau_C)).exp()
        rhos.append(rho)

    g = sum(rhos, C(0))
    posterior = [rho / g for rho in rhos]

    K = [C(0)] * 6
    K[0] = C(1)
    sqrt_tau = tau_C.sqrt()
    for n in range(1, 6):
        s = C(0)
        for (a, _), pi_i in zip(data, posterior):
            y = z - C(a)
            u = y / (2 * sqrt_tau)
            s += pi_i * hermite_phys(2 * n, u) / ((4 * tau_C) ** n)
        K[n] = s

    L = [C(0)] * 6
    # Passing Arb's analytic flag makes log reject invalid branch enclosures.
    L[0] = g.log(analytic=analytic)
    for n in range(1, 6):
        s = K[n]
        for r in range(1, n):
            s -= binomial(n - 1, r - 1) * L[r] * K[n - r]
        L[n] = s

    q = C(0)
    for r in range(0, 6):
        q += binomial(5, r) * K[r] * L[5 - r]
    return g * q


def integrate_real(a, b):
    """Rigorous real ball for integral_a^b G(x) dx."""
    I = C.integral(
        G,
        C(a),
        C(b),
        abs_tol=R("1e-30"),
        rel_tol=R("1e-30"),
        use_heap=True,
    )
    # The exact integral is real. This assertion catches failed analytic
    # continuation or insufficient subdivision.
    assert I.imag().contains_zero()
    assert I.imag().above_abs() < R("1e-25")
    return I.real()


# The displayed midpoints in Proposition 1, with exact rational intervals.
blocks = [
    (QQ(0), QQ(1) / 2, "161.9975"),
    (QQ(1) / 2, QQ(1), "-219.0117"),
    (QQ(1), QQ(3) / 2, "73.0107"),
    (QQ(3) / 2, QQ(2), "1.8755"),
    (QQ(2), QQ(5) / 2, "-49.7406"),
    (QQ(5) / 2, QQ(3), "78.9371"),
    (QQ(3), QQ(7) / 2, "-92.4682"),
    (QQ(7) / 2, QQ(4), "90.8184"),
    (QQ(4), QQ(9) / 2, "-79.1908"),
    (QQ(9) / 2, QQ(5), "58.8918"),
    (QQ(5), QQ(11) / 2, "-30.2433"),
    (QQ(11) / 2, QQ(6), "-3.3394"),
    (QQ(6), QQ(13) / 2, "29.4837"),
    (QQ(13) / 2, QQ(7), "-45.4866"),
    (QQ(7), QQ(15) / 2, "50.1159"),
    (QQ(15) / 2, QQ(8), "-43.9915"),
    (QQ(8), QQ(17) / 2, "29.8270"),
    (QQ(17) / 2, QQ(9), "-14.0523"),
    (QQ(9), QQ(10), "5.7303"),
    (QQ(10), QQ(12), "-3.1009"),
    (QQ(12), QQ(15), "0.3030"),
    (QQ(15), QQ(25), "-0.0000"),
]

finite_total = R(0)
for a, b, midpoint in blocks:
    block = 2 * integrate_real(a, b)
    finite_total += block
    displayed = R(midpoint)
    assert (block - displayed).above_abs() < R("1e-4")
    print(f"2 * integral_[{a}, {b}] G = {block}")

# Tail estimate from the proof:
#   |2 int_25^infty G(x) dx| <= 2 * 10^8 * exp(-128) / 16.
tail_bound = 2 * R("1e8") * (-R(128)).exp() / R(16)
assert tail_bound.upper() < R("1e-20").lower()

lower = finite_total.lower() - tail_bound.upper()
upper = finite_total.upper() + tail_bound.upper()

print("finite part =", finite_total)
print("tail bound  =", tail_bound)
print("certified H^(5)(1/3) interval = [", lower, ",", upper, "]")

assert lower > R("0.36").upper()
assert upper < R("0.37").lower()
\end{lstlisting}

\end{document}